\documentclass[10pt]{amsart}

\makeatletter
\def\blfootnote{\gdef\@thefnmark{}\@footnotetext}
\makeatother

\usepackage{epsfig}
\usepackage{graphics}
\usepackage{dcpic, pictexwd}

\usepackage[colorlinks=true,
                    linkcolor=red,
                    urlcolor=red,
                    citecolor=red,
                    anchorcolor=blue]{hyperref}
\usepackage{mathtools}
\usepackage{amsmath,amssymb}
\theoremstyle{plain}

\newtheorem*{theorem*}{Theorem}
\newtheorem*{thma}{Theorem A}
\newtheorem*{thmb}{Theorem B}
\newtheorem*{corc}{Corollary C}
\newtheorem{theorem}{Theorem}[section]
\theoremstyle{remark}

\theoremstyle{Acknowledgments}

\theoremstyle{definition}

\def\mod{{\rm Mod}}

\begin{document}
\blfootnote{\textup{2000} \textit{Mathematics Subject Classification}:
57N05, 20F38, 20F05}
\newenvironment{prooff}{\medskip \par \noindent {\it Proof}\ }{\hfill
$\square$ \medskip \par}
    \def\sqr#1#2{{\vcenter{\hrule height.#2pt
        \hbox{\vrule width.#2pt height#1pt \kern#1pt
            \vrule width.#2pt}\hrule height.#2pt}}}
    \def\square{\mathchoice\sqr67\sqr67\sqr{2.1}6\sqr{1.5}6}
\def\pf#1{\medskip \par \noindent {\it #1.}\ }
\def\endpf{\hfill $\square$ \medskip \par}
\def\demo#1{\medskip \par \noindent {\it #1.}\ }
\def\enddemo{\medskip \par}
\def\qed{~\hfill$\square$}

 \title[Generating the Mapping Class Group of a Nonorientable Surface] {Generating the Mapping Class Group of a Nonorientable Surface by Two Elements or By Three involutions}

\author[T{\"{u}}l\.{i}n Altun{\"{o}}z,       Mehmetc\.{i}k Pamuk, and O\u{g}uz Y{\i}ld{\i}z ]{T{\"{u}}l\.{i}n Altun{\"{o}}z,    Mehmetc\.{i}k Pamuk, and O\u{g}uz Y{\i}ld{\i}z}

\address{Department of Mathematics, Middle East Technical University,
 Ankara, Turkey}
\email{atulin@metu.edu.tr}  \email{mpamuk@metu.edu.tr} \email{oguzyildiz16@gmail.com}

%\subjclass{57M99, 20F38}
%\date{\today}
%\keywords{Mapping class groups, Lefschetz fibrations, Bounded cohomology}
%\thanks{* Supported by T\"UBA/GEB\.IP}

\begin{abstract}
We prove that, for $g\geq19$ the mapping class group of a nonorientable surface of genus $g$, $\mod(N_g)$, can be generated by two elements, one of which is of order $g$. We also prove that for $g\geq26$, $\mod(N_g)$ can be generated by three involutions if $g\geq26$. 
\end{abstract}

 \maketitle
%\tableofcontents
  \setcounter{secnumdepth}{2}
 \setcounter{section}{0}

\section{Introduction}

The mapping class group $\mod(N_g)$ of  closed connected nonorientable surface $N_g$  is defined to be the group of the isotopy classes of all self-diffeomorphisms
of $N_g$. In this paper, we are interested in finding  generating sets for $\mod(N_g)$ consisting of least possible number of elements. Since this group is not abelian, a generating set must contain at least two elements. %Chillingworth~\cite{c} obtained a finite collection of generators for $\mod(N_g)$. It is known that the mapping class group $\mod(N_g)$ must contain at least two elements.  
Szepietowski~\cite{sz2} proved that $\mod(N_{g})$ is generated by three elements for all $g\geq3$. %Since the abelianization of $\mod(N_{4})$ is $\mathbb{Z}_{2}^{3}$~\cite{mk3}, it is clear that $\mod(N_{4})$ cannot be generated by two elements.
Our first result (see Theorem~\ref{thm1}) answers Problem $3.1(a)$ in~\cite[p.$91$]{F} (cf Problem $5.4$ in~\cite{mk44}). 
\begin{thma}\label{thma}
For $g\geq 19$, the mapping class group $\mod(N_{g})$ is generated by two elements.
\end{thma}

 %Let us state that the strategy of our method forces the lower bound for the number of the genus in the above theorem. For $3\leq g \leq18$ (except for the case $g=4$), we still do not know if the minimal number of generators is two.

The next aim of the paper is to find an answer Problem $3.1(b)$ in~\cite[p.$91$]{F}. Szepietowski showed that $\mod(N_{g})$ can be generated by involutions~\cite{sz1} and later he showed that $\mod(N_{g})$ can be generated by four involutions if $g\geq4$~\cite{sz2}. One can deduce that it can be generated by three involutions by the work of Birman and Chillingworth~\cite{bc} if $g=3$. It is known that any group generated by two involutions is isomorphic to a quotient of a dihedral group. Thus the mapping class group $\mod(N_{g})$ cannot be generated by two involutions. This implies that any generating set consisting only involutions must contain at least three elements.  In this direction, we get the following result (see Theorem~\ref{teven} and Theorem~\ref{todd}):

\begin{thmb}\label{thmb}
For $g\geq 26$, the mapping class group $\mod(N_{g})$ can be generated by three involutions.
\end{thmb}
 %Again because of our method, the lower bound for the number of the genus in the above theorem is 26. For $4\leq g \leq25$, we still do not know if the minimal number of involution generators is three.

Let us also point out that $\mod(N_{g})$ admits an epimorphism onto the automorphism group of $H_1(N_g;\mathbb{Z}_2)$ preserving the $\pmod{2}$ intersection pairing~\cite{mpin} and this group is isomorphic to (see~\cite{mk3} and~\cite{sz3})
\begin{eqnarray}
 \begin{cases} Sp(2h;\mathbb{Z}_2)&\text{if $g=2h+1$,} \\ 
 Sp(2h;\mathbb{Z}_2)\ltimes \mathbb{Z}_{2}^{2h+1}&\text{if $g=2h+2$.}  \end{cases}
\nonumber 
\end{eqnarray}
Hence, the action of mapping classes on $H_1(N_g;\mathbb{Z}_2)$ induces an epimorphism from $\mod(N_{g})$ to 
$Sp\big(2\lfloor\dfrac{g-1}{2}\rfloor;\mathbb{Z}_2\big)$, which immediately implies the following corollary:
\begin{corc}
The symplectic group $Sp\big( g-1; \mathbb{Z}_2\big)$ can be generated by two elements for every odd $g\geq19$ and also by three involutions for every odd $g\geq27$. Similarly, the group $Sp\big(g-2; \mathbb{Z}_2\big) \ltimes \mathbb{Z}_{2}^{g-1}$ can be generated by two elements for every even $g\geq 20$ and also by three involutions for every even $g\geq26$.
\end{corc}
%Here is a brief history of...?

%The organization of the paper is as follows. In Section~\ref{S2}, we recall some basic results on $\mod(N_g)$. We  work with nonorientable surfaces of even genus in Section ~\ref{S3} and nonorientable surfaces of odd genus in Section ~\ref{S4}.
\medskip

\noindent
%{\bf Acknowledgments.}
\textit{Acknowledgments.} The first author was partially supported by the Scientific and 
Technologic Research Council of Turkey (TUBITAK)[grant number 120F118].

%%%%%%%%%%%%%%%%%%%%%%%%%%%%%%%%%%%%%%%%%%%%%%%%%%%%%%%%%%%%%%%%%%%%%%%%%%%%%%%%%%%%%%%%%%%%%%%%

%\section{ preliminaries}\label{S2}
\par  
\section{Preliminaries} \label{S2}
 Let $N_g$ be a closed connected nonorientable surface of genus $g$. 
 Note that the {\textit{genus}} for a nonorientable surface is the number 
 of projective planes in a connected sum decomposition. We use the model for the surface $N_g$ as a sphere with $g$ crosscaps represented shaded disks in all figures of this paper. Note that a crosscap is obtained by deleting the interior of such a disk and identifying the antipodal points on the resulting boundary. The {\textit{mapping class group}} 
 $\mod(N_g)$ of the surface $N_g$ is the group of the isotopy classes of self-diffeomorphisms of $N_g$. We
use the functional notation for the composition of two diffeomorphisms; if $f$ and $g$ are two diffeomorphisms, 
the composition $fg$ means that $g$ is applied first.

A simple closed curve on a nonorientable surface $N_g$ is 
\textit{one-sided} if its regular neighbourhood is 
a M\"{o}bius band and \textit{two-sided} if it is an annulus. If $a$ is a two-sided simple closed curve on $N_g$, to define the Dehn twist $t_a$ about the curve $a$, we need to choose one of two possible 
orientations of its regular neighbourhood (as we did for the curves in Figure~\ref{G}). Throughout the paper,  the right-handed 
Dehn twist $t_a$ about the curve $a$ will be denoted by the corresponding capital 
letter $A$. In our notation, both the curves on $N_g$ and  self-diffeomorphisms of $N_g$ shall be considered up to isotopy. In the following we shall make repeated use of some basic relations in $\mod(N_g)$: for two-sided simple closed curves $a$ and $b$ 
on $N_g$ and for any $f\in \mod(N_g)$,
\begin{itemize}
\item \textit{Commutativity:} If $a$ and $b$ are disjoint, then $AB=BA$.
\item \textit{Conjugation:} If $f(a)=b$, then $fAf^{-1}=B^{\varepsilon}$, where $\varepsilon=\pm 1$ 
depending on the orientation of a regular neighbourhood of $f(a)$ with respect to the chosen orientation.
\end{itemize}

\begin{figure}[h]
\begin{center}
\scalebox{0.6}{\includegraphics{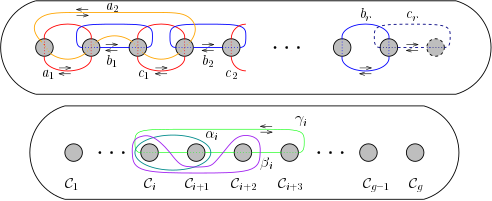}}
\caption{The curves $a_1,a_2,b_i,c_i,\alpha_i,\beta_i$ and $\gamma_i$ on the surface $N_g$, where $g=2r$ or $g=2r+2$. Note that we do not have the curve $c_r$ when $g$ is odd.}
\label{G}
\end{center}
\end{figure}

\begin{figure}[h]
\begin{center}
\scalebox{0.6}{\includegraphics{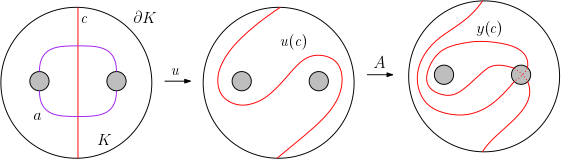}}
\caption{The homeomorphisms $u$ and $y=Au$.}
\label{Y}
\end{center}
\end{figure}
Consider the Klein bottle $K$ with a hole in Figure~\ref{Y}.
We define a \textit{crosscap transposition} $u$ as the isotopy classes of a diffeomorphism interchanging two consecutive crosscaps as shown on the left hand side of Figure~\ref{Y} and equals to the identity outside the Klein bottle with one hole $K$. The effect of the diffeomorphism $y=Au$ on the interval $c$
as in Figure~\ref{Y} can be also constructed as sliding a M{\"{o}}bius band once along the core of another one and keeping each point of the boundary of $K$ fixed. This is a \textit{$Y$-homeomorphism}~\cite{l1} (also called a \textit{crosscap slide}~\cite{mk4}). Note that $A^{-1}u$ is a $Y$-homemorphism i.e. the other choice of the orientation for a neighbourhood of the curve $a$ also gives a $Y$-homeomorphism. We also note that $y^{2}$ is a Dehn twist about $\partial K$.

It is known that $\mod(N_g)$ is generated by Dehn twists and a $Y$-homeomorphism (one crosscap slide)~\cite{l1}. We remark that crosscap transpositions can be used instead of crosscap slides since a crosscap transposition equals to the product of a Dehn twist and a crosscap slide.

Before we finish Preliminaries, let us state a theorem which is used in the proofs of following theorems. We work with the model in Figure~\ref{T} in such a way that the surface is obtained from the $2$-sphere by deleting the interiors of $g$ disjoint disks which are in a circular position and identifying the antipodal points on the boundary. Moreover, note that the rotation $T$ by
$\frac{2\pi}{g}$ about the $x$-axis maps the crosscap $\mathcal{C}_i$  to $\mathcal{C}_{i+1}$  for $i=1,\ldots,g-1$ and $\mathcal{C}_g$ to $\mathcal{C}_1$.
\begin{theorem}\label{thm2.1}
For $g\geq7$, the mapping class group $\mod(N_g)$ can be generated by the elements $T$, $A_1A_2^{-1}$, $B_1B_2^{-1}$, and a $Y$-homeomorphism (or a crosscap transposition).
\end{theorem}
\begin{proof}
Let $G$ be the subgroup of $\mod(N_g)$  generated by the set 
$\lbrace T,A_1A_{2}^{-1},B_1B_{2}^{-1} \rbrace$.
 Szepietowski~\cite[Theorem $3$]{sz2} showed that  $A_1,A_2,B_i$ and $C_i$ as shown in Figure~\ref{G}, together with a $Y$-homeomorphism generate $\mod(N_g)$. %Chillingworth~\cite{c} obtained a finite collection of generators for $\mod(N_g)$. Szepietowski~\cite{sz2} proved that each of the Chillingworth's Dehn twist generators can be obtained using the elements  $A_1,A_2,B_i$ and $C_i$ shown in Figure~\ref{G}, where $i=1,\ldots,r$.
Therefore, it is enough to prove that the elements  $A_1,A_2,B_i$ and $C_i$ are contained in $G$ for $i=1,\ldots,r$.

Let $\mathcal{S}$ denote the finite set of isotopy classes of two-sided non-separating 
simple closed curves appearing throughout the paper with chosen orientations of neighborhoods. Define a subset $\mathcal{G}$ of $\mathcal{S}\times \mathcal{S}$ 
as 
\[
\mathcal{G} =\lbrace(a,b): AB^{-1}\in G \rbrace.
\]
Using the similar arguments in the proof of ~\cite[Theorem $5$]{mk1}, the set $\mathcal{G}$ satisfies
\begin{itemize}
	\item if $(a,b)\in \mathcal{G}$, then $(b,a)\in \mathcal{G}$ (symmetry),
	\item if $(a,b) \ \textrm{and} \ (b,c)\in \mathcal{G}$, then $(a,c)\in \mathcal{G}$ (transitivity)  and
	\item if $(a,b)\in \mathcal{G}$ and $H\in G$ then $(H(a),H(b))\in \mathcal{G}$ ($G$-invariance).
	\end{itemize}
Thus, $\mathcal{G}$ defines an equivalence relation on $\mathcal{S}$. \\
\noindent
We begin by showing that $B_iC_{j}^{-1}$ is contained in $G$ for all $i,j$. It follows from the definition of $G$ and from the fact that $T(b_1,b_2)=(c_1,c_2)$, we have $C_1C_{2}^{-1}\in G$ (here, we use the notation $f(a,b)$ to denote $(f(a),f(b))$). Also, by conjugating  $C_1C_{2}^{-1}$ with powers of $T$, one can conclude that $G$ contains the elements $B_{i}B_{i+1}^{-1}$ and $C_{i}C_{i+1}^{-1}$. Moreover, the transitivity implies that the elements $B_{i}B_{j}^{-1}$ and $C_{i}C_{j}^{-1}$ are in $G$. To start with, since $B_2B_{3}^{-1}\in G$ and it is easy to verify that 
\[
B_2B_{3}^{-1}A_2A_{1}^{-1}(b_2,b_3)=(a_2,b_3),
\]
so that $A_2B_{3}^{-1} \in G$. Then, we have
\[
	 (A_1A_{2}^{-1})(A_2B_{3}^{-1})(B_3B_{2}^{-1})=A_{1}B_{2}^{-1}\in G,
\]
since $G$ contains each of the factors. Thus, $T(a_1,b_2)=(b_1,c_2)$ implies that $B_1C_{2}^{-1}$ is also in $G$. Moreover, $G$ contains the element
\[
B_1C_{1}^{-1}=(B_1C_{2}^{-1})(C_2C_{1}^{-1}).
\]
Thus,  $B_iC_{i}^{-1}\in G$ by conjugating with powers of $T$ for all $i=1,\ldots,r-1$. Again,the transitivity implies that $B_iC_{j}^{-1}\in G$. Note that, we have 
\begin{itemize}
	\item $(A_1B_{2}^{-1})(B_2C_{1}^{-1})=A_1C_{1}^{-1} \in G$,
	\item$(C_1A_{1}^{-1})(A_1A_{2}^{-1})=C_1A_{2}^{-1}\in G$ and
	\item$(C_2C_{1}^{-1})(C_1A_{1}^{-1})=C_2A_{1}^{-1}\in G$
	\end{itemize}
	from which it follows that the elements $A_1C_{1}^{-1}$, $C_1A_{2}^{-1}$ and $C_2A_{1}^{-1}$ are all in $G$.
	
It can also be verified that 
\[
(A_1B_{2}^{-1})(A_1C_{1}^{-1})(A_1C_{2}^{-1})(A_1B_2^{-1})(a_2,a_1)=(d_2,a_1)
\]	
so that $D_2A_1^{-1}\in G$. Also, the element $D_2C_2^{-1}=(D_2A_1^{-1})(A_1C_2^{-1})$ is in $G$. It can also be shown that
\[
(C_2B_{1}^{-1})(C_2A_{1}^{-1})(C_2C_{1}^{-1})(C_2B_{1}^{-1})(d_2,c_2)=(d_1,c_2),
\]
which implies that $G$ contains $D_1C_{2}^{-1}$. Thus, $G$ contains the element
\[
 D_1A_1^{-1}=(D_1C_2^{-1})(C_2A_1^{-1})
 \]
  (here, the curves $d_1$ and $d_2$ are shown in \cite[Figure $4$]{bk}).
By similar arguments as in the proof of ~\cite[Lemma $5$]{bk}, for $g\geq7$ the lantern relation implies that
\[
A_3=(A_2C_{1}^{-1})(D_1C_{2}^{-1})(D_2A_{1}^{-1} ).
\]
Since $G$ contains each factor on the right hand side,  $A_3\in G$. It follows from 
the diffeomorphism $A_3(B_3B_1^{-1})$ maps the curve $a_3$ to $b_3$ that
\[
B_3=A_3(B_3B_{1}^{-1})A_3(B_1B_{3}^{-1})A_{3}^{-1}\in G.
\]
By conjugating $B_3$ with the powers of $T$, we conclude that  $A_1,B_1,C_1,\ldots B_{r-1},C_{r-1}$ and $B_r$
are all in $G$. Moreover, 
\[
A_2=(A_2A_{1}^{-1})A_1 \in G.
\]
Therefore, the Dehn twist generators are contained in $G$. This finishes the proof.
\end{proof}

\section{A generating set for $\mod(N_g)$}\label{S3}
In this section, we work with the model in Figure~\ref{T}. Let us denote by $u_i$ the crosscap transposition supported on the one holed Klein bottle whose boundary is the curve $\alpha_i$ shown in Figure~\ref{G}. Note that the rotation $T$ takes $\alpha_i$ to $\alpha_{i+1}$ and the crosscap $\mathcal{C}_{i}$ to $\mathcal{C}_{i+1}$, which implies that $Tu_iT^{-1}=u_{i+1}$.

\begin{figure}[h]
\begin{center}
\scalebox{0.3}{\includegraphics{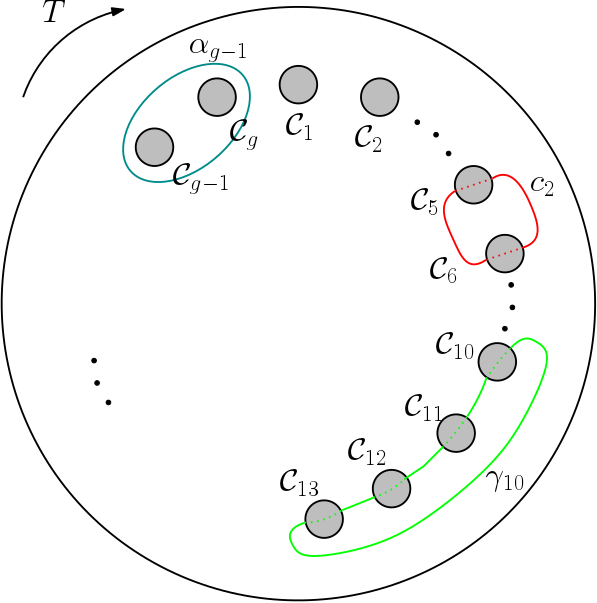}}
\caption{The rotation $T$ and the curves $c_2,\gamma_{10}$ and $\alpha_{g-1}$.}
\label{T}
\end{center}
\end{figure}

\begin{theorem}\label{thm1}
For $g\geq19$, the mapping class group $\mod(N_g)$ is generated by $\lbrace T,u_{g-1}\Gamma_{10}C_2^{-1} \rbrace$.
\end{theorem}
\begin{proof}
Let $F_1=u_{g-1}\Gamma_{10}C_2^{-1}$ and let us denote by $G$ the subgroup of $\mod(N_g)$ generated by $T$ and $F_1$. It follows from Theorem~\ref{thm2.1} that it suffices to prove that the subgroup $G$ contains the elements $A_1A_2^{-1}$, $B_1B_2^{-1}$ and $u_{g-1}$ to prove that $G=\mod(N_g)$.

Let $F_2$ denote the conjugation of $F_1$ by  $T^{-4}$.
It follows from $T^{-4}$ maps the curves $(\alpha_{g-1},\gamma_{10}, c_2)$ to $(\alpha_{g-5},\gamma_{6},a_1)$ that 
\[
F_2=T^{-4}F_1T^4=u_{g-5}\Gamma_{6}A_1^{-1}
\]
is contained in $G$. Let $F_3$ denote the element $(F_2F_1^{-1})F_2(F_2F_1^{-1})^{-1}$ that is contained in $G$. Hence
\[
F_3=(F_2F_1^{-1})F_2(F_2F_1^{-1})^{-1}=u_{g-5}C_2A_1^{-1}.
\]
Since we have similar cases in the remaining parts of the paper, let us give some details before we proceed. It can be verified that the diffeomorphism $F_2F_1^{-1}$ send the curves $(\alpha_{g-5},\gamma_6,a_1)$ to the curves $(\alpha_{g-5},c_2,a_1)$. Then, we get
\begin{eqnarray*}
F_3&=&(F_2F_1^{-1})F_2(F_2F_1^{-1})^{-1}\\
&=&(F_2F_1^{-1})u_{g-5}\Gamma_{6}A_1^{-1}(F_2F_1^{-1})^{-1}\\
&=&u_{g-5}C_2A_1^{-1}.
\end{eqnarray*}
Thus, we have the elements $F_2F_3^{-1}=\Gamma_{6}C_2^{-1}$ and $T^4(\Gamma_{6}C_2^{-1})T^{-4}=\Gamma_{10}C_4^{-1}$, which are both contained in $G$.

Moreover, we have the following elements
\begin{eqnarray*}
F_4&=&(C_4\Gamma_{10}^{-1})F_1=u_{g-1}C_4C_2^{-1},\\
F_5&=&T^{-1}F_4T=u_{g-2}B_4B_2^{-1} \textrm{ and }\\
F_6&=&(F_4F_5)F_3(F_4F_5)^{-1}=u_{g-5}B_2A_1^{-1},
\end{eqnarray*}
all of which are contained in the subgroup $G$. From this, we get the element $F_6F_3^{-1}=B_2C_2^{-1}\in G$. Also, we have  $T(B_2C_2^{-1})T^{-1}=C_2B_3^{-1}\in G$, which gives rise to
\[
B_2B_3^{-1}=(B_2C_2^{-1})(C_2B_3^{-1})\in G.
\]
This implies that $T^{-2}(B_2B_3^{-1})T^2=B_1B_2^{-1}$ is in $G$. We also have the elements
 \begin{eqnarray*}
 T^{2}(C_2B_3^{-1})T^{-2}&=&C_3B_4^{-1}\in G \textrm{ and }\\
 T^{-2}(\Gamma_{10}C_4^{-1})T^{2}&=&\Gamma_8C_3^{-1}\in G,
  \end{eqnarray*}
 implying that $\Gamma_8B_4^{-1}=(\Gamma_8C_3^{-1})(C_3B_4^{-1})\in G$. The conjugation of the element $\Gamma_8B_4^{-1}$ by $T^{-7}$ is the element $\Gamma_1A_1^{-1}=A_2A_1^{-1}$ which is contained in $G$. By the proof of Theorem~\ref{thm2.1}, the subgroup $G$ contains the elements $A_1$, $A_2$, $B_i$ and $C_i$ for $i=1,\ldots,r$. Then, in particular we have the elements $T^9A_2T^{-9}=\Gamma_{10}\in G$ and $C_2\in G$. We conclude that $u_{g-1}=F_1(C_2\Gamma_{10}^{-1})\in G$, which completes the proof.
 \end{proof}
 \section{Involution generators for $\mod(N_g)$}
In the first part of this section, where the genus of the surface $N_g$ is even, we refer to Figure~\ref{RE} for the involution generators $\rho_1$ and $\rho_2$ of $N_g$. The elements $\rho_1$ and $\rho_2$ are reflections about the indicated planes in Figure~\ref{RE} in such a way that the rotation $T$, depicted in Figure~\ref{T}, is given by $T=\rho_2\rho_1$. \begin{figure}[h]
\begin{center}
\scalebox{0.46}{\includegraphics{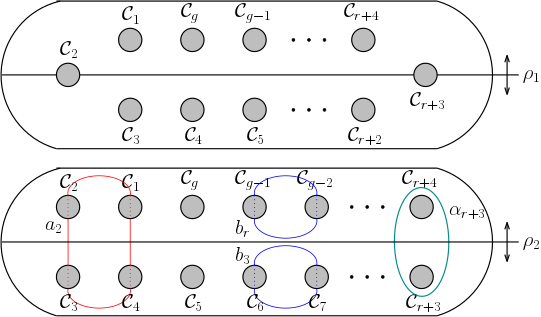}}
\caption{The reflections $\rho_1$ and $\rho_2$ for $g=2r+2$.}
\label{RE}
\end{center}
\end{figure}
\begin{theorem}\label{teven}
For $g=2r+2\geq26$, the mapping class group $\mod(N_g)$ is generated by the involutions $\rho_1$, $\rho_2$ and $\rho_2A_2B_rB_3u_{r+3}$.
\end{theorem}
\begin{proof}
Consider the surface $N_g$ as in Figure~\ref{RE}. It follows from
\[
\rho_2(a_2)=a_2 \textrm{ and }    \rho_2(b_r)=b_3
\]
and also $\rho_2$ reverses the given orientation of a neighbourhood of a two-sided simple closed curve that
\[
\rho_2A_2\rho_2=A_2^{-1} \textrm{ and } \rho_2B_r\rho_2=B_3^{-1}.
\]
Since $\rho_2u_{r+3}\rho_2=u_{r+3}^{-1}$, one can verify that the element $\rho_2A_2B_rB_3u_{r+3}$ is an involution. Let $H_1=A_2B_rB_3u_{r+3}$ and let $H$ be the subgroup of $\mod(N_g)$ generated by the set
\[
\lbrace \rho_1,\rho_2, \rho_2H_1\rbrace.
\]
It is clear that $H_1$ and $T=\rho_2\rho_1$ are contained in the subgroup $H$. By Theorem~\ref{thm2.1}, we need to prove that the subgroup $H$ contains the elements $A_1A_2^{-1}$, $B_1B_2^{-1}$ and $u_{r+3}$. Let $H_2$ be the conjugation of $H_1$ by $T^7$. Thus
\[
H_2=T^7H_1T^{-7}=\Gamma_8C_2C_6u_{r+10} \in H.
\]
Let 
\[
H_3=(H_2H_1)H_2(H_2H_1)^{-1}=\Gamma_8B_3C_6u_{r+10},
\]
which is also in $H$. From this, we get the element $H_2H_3^{-1}=C_2B_3^{-1}\in H$ implying that $T(C_2B_3^{-1})T^{-1}=B_3C_3^{-1} \in H$. One can easily see that  $B_iC_i^{-1} \in H$ by conjugating $B_3C_3^{-1}$ with powers of $T$. Also, since $T(B_3C_3^{-1})T^{-1}=C_3B_4^{-1} \in H$, similarly $C_{i}B_{i+1}^{-1}\in H$ by conjugating $C_3B_4^{-1}$ with powers of $T$. Hence, we have the elements
\[
B_iB_{i+1}^{-1}=(B_iC_i^{-1})(C_{i}B_{i+1}^{-1})
\]
which are in $H$ for all $i=1,\ldots,r-1$. Moreover, $B_iB_j^{-1}\in H$ by the transitivity. In particular $B_1B_2^{-1}\in H$. Now, we have the following elements
\begin{eqnarray*}
H_4&=&(B_7B_3^{-1})H_1=A_2B_7B_ru_{r+3} \textrm{ if }r\neq 16,17,18,\\
(H_4&=&(B_9B_3^{-1})H_1=A_2B_9B_ru_{r+3} \textrm{ if }r=16,17,18,)\\
H_5&=&T^{6}H_4T^{-6}=\Gamma_{7} B_{10}B_2u_{r+9}\textrm{ if }r\neq 16,17,18,\\
(H_5&=&T^{6}H_4T^{-6}=\Gamma_{7} B_{12}B_2u_{r+9} \textrm{ if }r=16,17,18,)\\
H_6&=&(H_5H_4)H_5(H_5H_4)^{-1}=\Gamma_{7} B_{10}A_2u_{r+9}\textrm{ if }r\neq 16,17,18,\\
(H_6&=&(H_5H_4)H_5(H_5H_4)^{-1}=\Gamma_{7} B_{12}A_2u_{r+9}\textrm{ if }r=16,17,18,)
\end{eqnarray*}
which are all contained in $H$. Thus, we get the element $H_6H_5^{-1}=A_2B_2^{-1}\in H$. On the other hand, since $C_1B_2^{-1}$ is contained in $H$, the subgroup $H$ contains the following elements
\[
T^{-2}(C_1B_2^{-1})T^2=A_1B_1^{-1},
\]
\[
(A_1B_1^{-1})(B_1B_2^{-1})=A_1B_2^{-1},
\]
\[
(A_2B_2^{-1})(B_2A_1^{-1})=A_2A_1^{-1}.
\]
It follows from $T$, $A_1A_2^{-1}$ and $B_1B_2^{-1}$ are in $H$ that the Dehn twists $A_1$, $A_2$, $B_i$ and $C_i$ are also in $H$ for $i=1,\ldots,r$. This implies that 
\[
u_{r+3}=(B_3^{-1}B_r^{-1}A_2^{-1})H_1\in H,
\]
which completes the proof.
\end{proof}
\begin{figure}[h]
\begin{center}
\scalebox{0.3}{\includegraphics{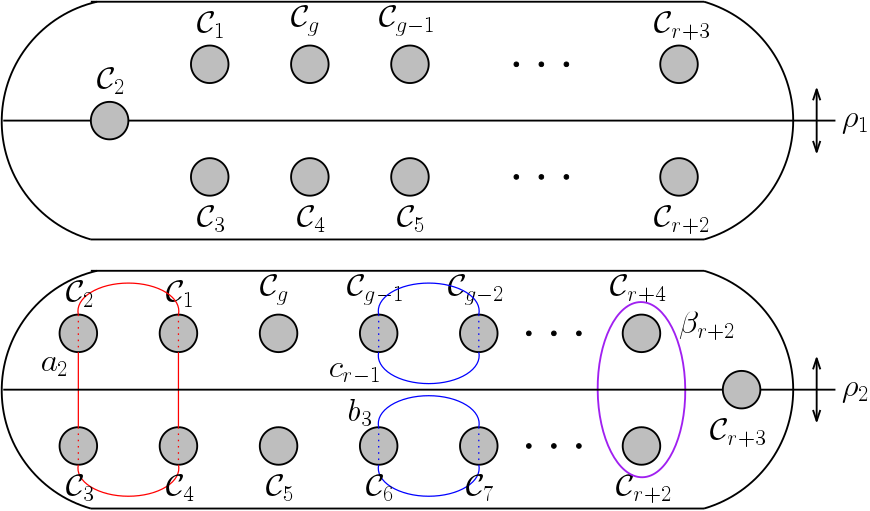}}
\caption{The reflections $\rho_1$ and $\rho_2$ for $g=2r+1$.}
\label{RO}
\end{center}
\end{figure}
In the second part of this section, where the genus of the surface $N_g$ is odd, we refer to Figure~\ref{RO} for the involution generators $\rho_1$ and $\rho_2$ of $N_g$. Similarly, the elements $\rho_1$ and $\rho_2$ are reflections about the indicated planes in Figure~\ref{RO} such that the rotation $T$ in Figure~\ref{T} is given by $T=\rho_2\rho_1$. In the proof of the following theorem, we use the crosscap transposition supported on the one holed Klein bottle whose boundary is the curve $\beta_i$ shown in Figure~\ref{G}. Let us denote this crosscap transposition by $v_i$.
Note that the rotation $T$ sends $\beta_i$ to $\beta_{i+1}$ and the crosscap $\mathcal{C}_i$ to $\mathcal{C}_{i+1}$, which implies that $Tv_iT^{-1}=v_{i+1}$.
\begin{theorem}\label{todd}
For $g=2r+1\geq27$, the mapping class group $\mod(N_g)$ is generated by the involutions $\rho_1$, $\rho_2$ and $\rho_2A_2C_{r-1}B_3v_{r+2}$.
\end{theorem}
\begin{proof}
We will follow the proof of Theorem~\ref{teven}, closely. Let us consider the surface $N_g$ as in Figure~\ref{RO}. Since
\[
\rho_2(a_2)=a_2 \textrm{ and }    \rho_2(c_{r-1})=b_3
\]
and also since $\rho_2$ reverses the given orientation of a neighbourhood of a two-sided simple closed curve, we get
\[
\rho_2A_2\rho_2=A_2^{-1} \textrm{ and } \rho_2C_{r-1}\rho_2=B_3^{-1}.
\]
By the fact that $\rho_2v_{r+2}\rho_2=v_{r+2}^{-1}$, it can be easy to verify that the element $\rho_2A_2C_{r-1}B_3\phi_{r+2,r+4}$ is an involution. Let $E_1=A_2C_{r-1}B_3v_{r+2}$ and let $K$ denote the subgroup of $\mod(N_g)$ generated by the set
\[
\lbrace \rho_1,\rho_2, \rho_2E_1\rbrace.
\]
It is easy to see that $E_1$ and $T=\rho_2\rho_1$ are in $K$. By Theorem~\ref{thm2.1}, we need to show that $K$ contains the elements $A_1A_2^{-1}$, $B_1B_2^{-1}$ and $v_{r+2}$. Let $E_2$ be the following:
\[
E_2=T^7E_1T^{-7}=\Gamma_8C_2C_6v_{r+9} \in K.
\]
Consider the element
\[
E_3=(E_2E_1)E_2(E_2E_1)^{-1}=\Gamma_8B_3C_6v_{r+9},
\]
which belongs to $K$. One can conclude that the element $E_2E_3^{-1}=C_2B_3^{-1}\in K$, which implies that $T(C_2B_3^{-1})T^{-1}=B_3C_3^{-1} \in K$. From this, we get the elements $B_iC_i^{-1} \in H$ by conjugating $B_3C_3^{-1}$ with powers of $T$. Also, since $T(B_3C_3^{-1})T^{-1}=C_3B_4^{-1} \in K$, $C_{i}B_{i+1}^{-1}\in K$ by again conjugating $C_3B_4^{-1}$ with powers of $T$. Thus, we get the elements
\[
B_iB_{i+1}^{-1}=(B_iC_i^{-1})(C_{i}B_{i+1}^{-1}),
\]
which belong to $K$ for all $i=1,\ldots,r-1$. Also, using the transitivity $B_iB_j^{-1}\in K$. In particular $B_1B_2^{-1}\in K$. Moreover, we have the elements
\begin{eqnarray*}
E_4&=&(B_7B_3^{-1})E_1=A_2B_7C_{r-1}v_{r+2} \textrm{ if }r\neq 16,17,18,19,\\
(E_4&=&(B_9B_3^{-1})E_1=A_2B_9C_{r-1}v_{r+2} \textrm{ if }r=16,17,18,19,)\\
E_5&=&T^{6}E_4T^{-6}=\Gamma_{7} B_{10}B_2v_{r+8}\textrm{ if }r\neq 16,17,18,19,\\
(E_5&=&T^{6}E_4T^{-6}=\Gamma_{7} B_{12}B_2v_{r+8} \textrm{ if }r=16,17,18,19,)\\
E_6&=&(E_5E_4)E_5(E_5E_4)^{-1}=\Gamma_{7} B_{10}A_2v_{r+8}\textrm{ if }r\neq 16,17,18,19,\\
(E_6&=&(E_5E_4)E_5(E_5E_4)^{-1}=\Gamma_{7} B_{12}A_2v_{r+8}\textrm{ if }r=16,17,18,19,)
\end{eqnarray*}
which are all contained in the subgroup $K$. Thus, we conclude that the element $E_6E_5^{-1}=A_2B_2^{-1}\in K$.

Since the element $C_1B_2^{-1}\in K$, as in the proof of Theorem~\ref{teven}, one can conclude that the Dehn twists $A_1$, $A_2$, $B_i$ and $C_j$ are in $K$ for $i=1,\ldots,r$ and $j=1,\ldots,r-1$. This implies that $v_{r+2}=(B_3^{-1}C_{r-1}^{-1}A_2^{-1})E_1\in K$, which finishes the proof.
%the subgroup $K$ contains the following elements:
%\[
%T^{-2}(C_1B_2^{-1})T^2=A_1B_1^{-1},
%\]
%\[
%(A_1B_1^{-1})(B_1B_2^{-1})=A_1B_2^{-1},
%\]
%\[
%(A_2B_2^{-1})(B_2A_1^{-1})=A_2A_1^{-1}.
%\]
%Therefore, since the subgroup $K$ contains the elements $T$, $A_1A_2^{-1}$ and $B_1B_2^{-1}$, the proof of Theorem~\ref{thm2.1} implies that 
\end{proof}

%%%%%%%%%%%%%%%%%%%%%%%%%%%%%%%%%%%%%%%%%%%%%%%%%%%%%%%%%%%%%%%%%%%%

\end{document}